\documentclass[11pt]{article}
\usepackage{cmap}
\usepackage[T1]{fontenc}
\usepackage[utf8]{inputenc}
\usepackage{microtype}
\usepackage[english]{babel}
\usepackage[a4paper,margin=29mm]{geometry}
\usepackage{amsmath,amssymb,amsthm,mathtools}
\usepackage{enumitem}
\usepackage{xcolor}
\usepackage{indentfirst}
\usepackage[numbers,sort&compress]{natbib}
\usepackage[colorlinks=true,linkcolor=blue!45!black,citecolor=blue!45!black,urlcolor=blue!45!black]{hyperref}
\allowdisplaybreaks[2]
\newcommand{\bzeta}{\boldsymbol{\zeta}}
\newcommand{\betta}{\boldsymbol{\eta}}

\newtheorem{theorem}{Theorem}[section]
\newtheorem{proposition}{Proposition}[section]

\theoremstyle{remark}
\newtheorem{remark}{Remark}[section]

\author{Kirill Bakhtin\thanks{Institute for Applied Mathematics, Far Eastern Branch, Russian Academy of Sciences (IAM FEB RAS), 7~Radio Street, Vladivostok 690041, Russia.}\textsuperscript{,}\thanks{Far Eastern Center of Mathematical Research, Far Eastern Federal University, 10~Ajax Bay, Russky Island, Vladivostok 690922, Russia.}}
\title{On transformations of Kamp\'e de F\'eriet functions with integral parameter differences}
\date{}

\begin{document}

\maketitle

\begin{abstract}
We obtain analogues of the first Miller--Paris transformation for three classes of Kamp\'e de F\'eriet functions with unit parameter differences associated with the Appell function $F_1$. We also derive representations of the Appell function $F_1$ in terms of Kamp\'e de F\'eriet functions and reduction formulas expressing certain Kamp\'e de F\'eriet functions with integral parameter differences in terms of univariate generalized hypergeometric functions with integral parameter differences.
\end{abstract}

\noindent\textbf{Keywords.} hypergeometric function; Miller--Paris transformations; integral parameter differences; Appell function; Kamp\'e de F\'eriet function.

\section*{Introduction}

We use the standard notation $\mathbb{C}$, $\mathbb{Z}$, and $\mathbb{N}$ for the sets of complex numbers, integers, and positive integers, respectively, and put $\mathbb{N}_0=\mathbb{N}\cup\{0\}$. Scalars are denoted by ordinary letters and vectors by bold symbols, $\mathbf{a}=(a_1,\dots,a_p)$, $\mathbf{b}=(b_1,\dots,b_q)$. The Pochhammer symbol is defined by $(a)_n=a(a+1)\dots(a+n-1)$, $n\in\mathbb{N}$, $(a)_0=1$, and the falling factorial by $[a]_n=a(a-1)\dots(a-n+1)$, $[a]_0=1$. We use the abbreviations $(\mathbf{a})_n=(a_1)_n\cdots(a_p)_n$ and $\mathbf{a}+\mu=(a_1+\mu,\dots,a_p+\mu)$. Throughout the paper, $\mathbf f=(f_1,\dots,f_p)\in\mathbb C^p$. The elementary symmetric polynomials in $f_1,\dots,f_p$ are denoted by $\sigma_k(\mathbf f)$,
\begin{equation}\label{sigma}
\sigma_k(\mathbf f)=\sum_{1\le i_1<\cdots<i_k\le p}f_{i_1}\cdots f_{i_k},
\end{equation}
where $\sigma_0(\mathbf f)=1$ and $\sigma_p(\mathbf f)=\prod_{j=1}^p f_j=(\mathbf f)_1$.

The generalized hypergeometric function is defined by the series
\begin{equation}\label{eq:hypergeometic}
{}_{p}F_{q}\left(\begin{matrix}\mathbf a\\ \mathbf b\end{matrix}\Big|z\right)
=\sum_{n=0}^{\infty}\frac{(\mathbf a)_n}{(\mathbf b)_n\,n!}z^n,
\end{equation}
where $\mathbf a\in\mathbb C^p$, $\mathbf b\in\mathbb C^q$, and $b_j\notin-\mathbb N_0$ for $j=1,\dots,q$. For further information on generalized hypergeometric functions, see, for example, \cite{AAR,BealsWong,LukeBook}.

Generalized hypergeometric functions containing parameter pairs $\begin{bmatrix}f+m\\f\end{bmatrix}$, $m\in\mathbb N$, are called hypergeometric functions with integral parameter differences (IPD). The study of these functions goes back to the work of B.~Minton and P.~W.~Karlsson. Minton obtained a summation formula for a terminating generalized hypergeometric series at unit argument when one of the upper parameters is $-k$, $k\in\mathbb N$ \cite{Minton1970}. Karlsson generalized this formula by replacing $-k$ with an arbitrary complex parameter $a$ under an appropriate convergence condition and, in the same paper, obtained a reduction formula for hypergeometric functions with positive integral parameter differences \cite{Karlsson1971}. Simpler proofs of this reduction formula were given by Srivastava \cite{Srivastava1973}. For negative integral parameter differences, Karlsson represented the generalized hypergeometric function as a linear combination of hypergeometric functions of lower order \cite{Karlsson1974}; further generalizations were considered by Panda \cite{Panda1976}. Gasper obtained $q$-analogues of the Minton--Karlsson formulas \cite{Gasper1981}, while Chu derived corresponding formulas for bilateral hypergeometric series \cite{Chu1994,Chu1995}. Letessier, Valent, and Wimp studied differential equations satisfied by generalized hypergeometric functions with unit positive parameter differences \cite{LetessierValentWimp1994}. Thomae transformations for hypergeometric functions with positive integral parameter differences were obtained in \cite{KRP2014}; other summation and transformation formulas were considered, in particular, in \cite{KarpPrilepkina2018,my,my2}.

A parameter pair with an arbitrary positive integral difference can be represented by pairs with unit differences, since
\begin{equation}\label{dobavka}
\frac{(f+m)_n}{(f)_n}
=\frac{(f+1)_n(f+2)_n\cdots(f+m)_n}{(f)_n(f+1)_n\cdots(f+m-1)_n}.
\end{equation}
We therefore consider unit parameter differences below. Hypergeometric functions with integral parameter differences can also be interpreted in terms of polynomial perturbations. This interpretation of the Miller--Paris transformations is systematically developed in \cite{Karp2026}.

Miller and Paris obtained two extensions of the Euler--Pfaff transformations (see, for example, \cite{AAR}) to generalized hypergeometric functions with additional parameter pairs \cite{MPEuler2011,MP}. The first transformation, whose analogues are considered in this paper, takes the following form in the case of unit parameter differences:
\begin{equation}\label{MP1}
{}_{p+2}F_{p+1}\left(\begin{matrix}a,b,\mathbf f+\mathbf 1\\c,\mathbf f\end{matrix}\Big|x\right)
=(1-x)^{-a}
{}_{p+2}F_{p+1}\left(\begin{matrix}a,\lambda,\bzeta+\mathbf 1\\c,\bzeta\end{matrix}\Big|\frac{x}{x-1}\right),
\end{equation}
where $\lambda=c-b-p$, $\mathbf 1=(1,\dots,1)\in\mathbb N^p$, and $\bzeta=(\zeta_1,\dots,\zeta_p)$ is the vector of roots of the polynomial
\begin{equation}\label{MP1Pol}
M_p(t)=\sum_{k=0}^{p}\sigma_{p-k}(\mathbf f)
\sum_{l=0}^{k}S(k,l)(b)_l(t)_l(\lambda-t)_{p-l}.
\end{equation}
Here $S(k,l)$ are the Stirling numbers of the second kind,
\[
x^k=\sum_{l=0}^{k}S(k,l)[x]_l.
\]
Other forms of the Miller--Paris transformations and their characteristic polynomials can be found in \cite{MP,KarpPrilepkina2018}.

Throughout the paper, whenever a lower parameter in a transformed IPD pair is a negative integer, the corresponding ratio is understood in the polynomial sense
\[
\frac{(\xi+1)_r}{(\xi)_r}=\frac{\xi+r}{\xi},
\qquad \xi\ne0,\quad r\in\mathbb N_0.
\]
For several IPD pairs, we take the product of these polynomial factors. Thus, the upper and lower parameters in each pair remain linked by their unit difference; the series is not separately truncated at an upper parameter that is a nonpositive integer.

The transformations considered below are based on the Appell function $F_1$,
\begin{equation}\label{eq:F1}
F_1(a;b_1,b_2;c;x,y)
=\sum_{m=0}^{\infty}\sum_{n=0}^{\infty}
\frac{(a)_{m+n}(b_1)_m(b_2)_n}{(c)_{m+n}\,m!\,n!}x^m y^n,
\qquad |x|<1,\quad |y|<1,
\end{equation}
see \cite{AppellKF,Sh}. The function $F_1$ is symmetric under interchange of the pairs $(b_1,x)$ and $(b_2,y)$. In particular,
\[
F_1(a;b_1,0;c;x,y)
={}_2F_1\left(\begin{matrix}a,b_1\\c\end{matrix}\Big|x\right),
\qquad
F_1(a;0,b_2;c;x,y)
={}_2F_1\left(\begin{matrix}a,b_2\\c\end{matrix}\Big|y\right).
\]

The following analogues of the Euler--Pfaff transformations are known for $F_1$:
\begin{equation}\label{Sh24}
F_1(a;b_1,b_2;c;x,y)
=(1-x)^{-b_1}(1-y)^{-b_2}
F_1\left(c-a;b_1,b_2;c;\frac{x}{x-1},\frac{y}{y-1}\right),
\end{equation}
\begin{equation}\label{Sh25}
F_1(a;b_1,b_2;c;x,y)
=(1-x)^{-a}
F_1\left(a;c-b_1-b_2,b_2;c;\frac{x}{x-1},\frac{y-x}{1-x}\right),
\end{equation}
\begin{equation}\label{Sh27}
F_1(a;b_1,b_2;c;x,y)
=(1-y)^{-a}
F_1\left(a;b_1,c-b_1-b_2;c;\frac{x-y}{1-y},\frac{y}{y-1}\right),
\end{equation}
\begin{equation}\label{Sh28}
F_1(a;b_1,b_2;c;x,y)
=(1-x)^{c-a-b_1}(1-y)^{-b_2}
F_1\left(c-a;c-b_1-b_2,b_2;c;x,\frac{x-y}{1-y}\right),
\end{equation}
\begin{equation}\label{Sh29}
F_1(a;b_1,b_2;c;x,y)
=(1-x)^{-b_1}(1-y)^{c-a-b_2}
F_1\left(c-a;b_1,c-b_1-b_2;c;\frac{y-x}{1-x},y\right).
\end{equation}
For $b_2=0$, formulas \eqref{Sh24} and \eqref{Sh25} yield the Pfaff transformations, while \eqref{Sh28} yields Euler's transformation. Similarly, for $b_1=0$, the Pfaff transformations follow from \eqref{Sh24} and \eqref{Sh27}, and Euler's transformation from \eqref{Sh29}.

Appell functions continue to be studied in contemporary research. Alongside transformation formulas, numerical summation and error estimation are also considered; for example, the recent paper \cite{Bezrodnykh} provides estimates for the remainder term in the summation of the Appell series $F_2$.

Adding IPD pairs to the parameter groups associated with $m+n$, $m$, or $n$ gives particular cases of the Kamp\'e de F\'eriet functions,
\begin{equation}\label{eq:KDF}
F^{\,p:k;l}_{\,q:r;s}\!\left(
\begin{array}{c}
\mathbf a:\mathbf b;\mathbf c\\
\mathbf d:\mathbf e;\mathbf g
\end{array}
\ ;x,y\right)
=\sum_{m=0}^{\infty}\sum_{n=0}^{\infty}
\frac{\prod_{i=1}^{p}(a_i)_{m+n}\prod_{j=1}^{k}(b_j)_m\prod_{t=1}^{l}(c_t)_n}
{\prod_{u=1}^{q}(d_u)_{m+n}\prod_{v=1}^{r}(e_v)_m\prod_{w=1}^{s}(g_w)_n}
\frac{x^m}{m!}\frac{y^n}{n!}.
\end{equation}
For information on Kamp\'e de F\'eriet functions, see \cite{AppellKF,Exton,SrivastavaKarlsson}. Reduction formulas for Kamp\'e de F\'eriet functions and their particular cases were considered, in particular, in \cite{Karlsson1984,CvijovicMiller2010,LiuWang2014,MillerParis2012}.

In this paper, we obtain analogues of the first Miller--Paris transformation for three classes of Kamp\'e de F\'eriet functions with unit parameter differences in the parameter groups associated with $m+n$, $m$, and $n$. We also derive representations of the Appell function $F_1$ in terms of Kamp\'e de F\'eriet functions and reduction formulas expressing certain Kamp\'e de F\'eriet functions with integral parameter differences in terms of univariate generalized hypergeometric functions.

All identities involving hypergeometric functions and double hypergeometric series are first considered in a nonempty domain where all the series involved converge absolutely. Outside this domain, they are understood by analytic continuation.

\section{Analogues of the first Miller--Paris transformation for Kamp\'e de F\'eriet functions}

In this section, IPD pairs are added in turn to the parameter groups associated with $m+n$, $m$, and $n$. In each case, the corresponding IPD factor is a polynomial in the relevant index. Expanding this polynomial in falling factorials expresses the Kamp\'e de F\'eriet function as a finite linear combination of $F_1$ functions with shifted parameters, to which we apply transformations \eqref{Sh24}, \eqref{Sh25}, or \eqref{Sh27}.

\begin{theorem}\label{thm:mn-transform}
Let $p\in\mathbb N$, $\mathbf f=(f_1,\dots,f_p)\in\mathbb C^p$, $a,b_1,b_2,c,x,y\in\mathbb C$, $c,f_1,\dots,f_p\notin-\mathbb N_0$, $a\ne f_j$ for $j=1,\dots,p$, $\lambda=c-a-p$, and $(\lambda)_p\ne0$. Then
\begin{multline}\label{KFTH1}
F_{p+1:0;0}^{p+1:1;1}\!\left(
\begin{array}{c}
a,\mathbf f+\mathbf1:b_1;b_2\\
c,\mathbf f:-;-
\end{array};x,y\right)\\
=(1-x)^{-b_1}(1-y)^{-b_2}
F_{p+1:0;0}^{p+1:1;1}\!\left(
\begin{array}{c}
\lambda,\bzeta+\mathbf1:b_1;b_2\\
c,\bzeta:-;-
\end{array};\frac{x}{x-1},\frac{y}{y-1}\right),
\end{multline}
where $\bzeta=(\zeta_1,\dots,\zeta_p)$ is the vector of roots of the polynomial
\begin{equation}\label{KFTH1Pol}
Q_p(t)=\sum_{k=0}^{p}\sigma_{p-k}(\mathbf f)\sum_{l=0}^{k}S(k,l)(a)_l(\lambda-t)_{p-l}(t)_l.
\end{equation}
\end{theorem}

\begin{proof}
Expanding the left-hand side of \eqref{KFTH1} by definition, we obtain
\begin{equation}\label{eq:mn-start}
F_{p+1:0;0}^{p+1:1;1}\!\left(
\begin{array}{c}
a,\mathbf f+\mathbf1:b_1;b_2\\
c,\mathbf f:-;-
\end{array};x,y\right)=\sum_{m=0}^{\infty}\sum_{n=0}^{\infty}
\frac{(a)_{m+n}(\mathbf f+\mathbf1)_{m+n}(b_1)_m(b_2)_n}
{(c)_{m+n}(\mathbf f)_{m+n}}\frac{x^m}{m!}\frac{y^n}{n!}.
\end{equation}
For the IPD factor, we have
\[
\frac{(\mathbf f+\mathbf1)_{m+n}}{(\mathbf f)_{m+n}}
=\frac{1}{(\mathbf f)_1}\prod_{j=1}^{p}(f_j+m+n)
=\frac{1}{(\mathbf f)_1}\sum_{k=0}^{p}\sigma_{p-k}(\mathbf f)(m+n)^k.
\]
Using the expansion of a power in falling factorials and the formula for the falling factorial of a sum,
\[
(m+n)^k=\sum_{l=0}^{k}S(k,l)[m+n]_l,
\qquad
[m+n]_l=\sum_{i=0}^{l}\binom li[m]_i[n]_{l-i},
\]
we obtain
\begin{equation}\label{frac}
\frac{(\mathbf f+\mathbf1)_{m+n}}{(\mathbf f)_{m+n}}
=\frac{1}{(\mathbf f)_1}
\sum_{k=0}^{p}\sigma_{p-k}(\mathbf f)
\sum_{l=0}^{k}S(k,l)
\sum_{i=0}^{l}\binom li[m]_i[n]_{l-i}.
\end{equation}
Substitute \eqref{frac} into \eqref{eq:mn-start}. Since $[m]_i/m!=1/(m-i)!$ and $[n]_{l-i}/n!=1/(n-l+i)!$, the inner sums start at $m=i$ and $n=l-i$. Making the substitutions $m\mapsto m+i$, $n\mapsto n+l-i$ and using the identity $(u)_{r+s}=(u)_r(u+r)_s$, we obtain
\begin{multline}\label{KDtoAppel}
F_{p+1:0;0}^{p+1:1;1}\!\left(
\begin{array}{c}
a,\mathbf f+\mathbf1:b_1;b_2\\
c,\mathbf f:-;-
\end{array};x,y\right)\\
=\frac{1}{(\mathbf f)_1}\sum_{k=0}^{p}\sigma_{p-k}(\mathbf f)
\sum_{l=0}^{k}S(k,l)\sum_{i=0}^{l}\binom li
\frac{(a)_l(b_1)_i(b_2)_{l-i}}{(c)_l}x^iy^{l-i}\\
\times F_1(a+l;b_1+i,b_2+l-i;c+l;x,y).
\end{multline}

We now apply transformation \eqref{Sh24} to each $F_1$ function:
\begin{multline}\label{eq:mn-Sh24}
F_1(a+l;b_1+i,b_2+l-i;c+l;x,y)\\
=(1-x)^{-b_1-i}(1-y)^{-b_2-l+i}
F_1\!\left(c-a;b_1+i,b_2+l-i;c+l;
\frac{x}{x-1},\frac{y}{y-1}\right).
\end{multline}
Expand the last function into a double series with summation indices $u,v\ge0$ and, in each term of the finite sum, set $m=u+i$ and $n=v+l-i$. We use the identities
\[
(m-i)!=(-1)^i\frac{m!}{(-m)_i},\qquad
(n-l+i)!=(-1)^{l-i}\frac{n!}{(-n)_{l-i}},
\]
and, with $\lambda=c-a-p$,
\begin{equation}\label{eq:mn-lambda-shift}
(c-a)_{m+n-l}
=\frac{(\lambda)_{m+n}(\lambda+m+n)_{p-l}}{(\lambda)_p}.
\end{equation}
After cancellation of the factors depending on $i$ and $l$, formula \eqref{KDtoAppel} takes the form
\begin{multline}\label{KFTH111}
F_{p+1:0;0}^{p+1:1;1}\!\left(
\begin{array}{c}
a,\mathbf f+\mathbf1:b_1;b_2\\
c,\mathbf f:-;-
\end{array};x,y\right)\\
=(1-x)^{-b_1}(1-y)^{-b_2}
\sum_{m=0}^{\infty}\sum_{n=0}^{\infty}
\frac{(\lambda)_{m+n}(b_1)_m(b_2)_n}{(c)_{m+n}m!n!}
\left(\frac{x}{x-1}\right)^m
\left(\frac{y}{y-1}\right)^n\\
\times\frac{1}{(\lambda)_p(\mathbf f)_1}
\sum_{k=0}^{p}\sigma_{p-k}(\mathbf f)
\sum_{l=0}^{k}S(k,l)(a)_l(\lambda+m+n)_{p-l}
\sum_{i=0}^{l}\binom li(-m)_i(-n)_{l-i}.
\end{multline}
By Vandermonde's convolution for the Pochhammer symbols,
\[
\sum_{i=0}^{l}\binom li(-m)_i(-n)_{l-i}=(-m-n)_l.
\]
Put $r=m+n$. After applying Vandermonde's convolution,
the polynomial factor in \eqref{KFTH111} can be written as

\begin{multline}
\frac{1}{(\lambda)_p(\mathbf f)_1}
\sum_{k=0}^{p}\sigma_{p-k}(\mathbf f)
\sum_{l=0}^{k}S(k,l)(a)_l
(\lambda+m+n)_{p-l}(-m-n)_l
=\\
\frac{1}{(\lambda)_p(\mathbf f)_1}
\sum_{k=0}^{p}\sigma_{p-k}(\mathbf f)
\sum_{l=0}^{k}S(k,l)(a)_l
(\lambda+r)_{p-l}(-r)_l
=
\frac{Q_p(-r)}{(\lambda)_p(\mathbf f)_1}.
\end{multline}
The identification with $Q_p(-r)$ follows directly from
\eqref{KFTH1Pol} by setting $t=-r$.
At $t=0$, only the term with $k=l=0$ remains in \eqref{KFTH1Pol}, and hence
\[
Q_p(0)=(\lambda)_p(\mathbf f)_1.
\]
If $\zeta_1,\dots,\zeta_p$ are the roots of $Q_p$, then
\begin{equation}\label{eq:Q-factor-mn}
\frac{Q_p(-r)}{Q_p(0)}
=\prod_{j=1}^{p}\frac{r+\zeta_j}{\zeta_j}
=\frac{(\bzeta+\mathbf1)_r}{(\bzeta)_r},
\qquad r\in\mathbb N_0.
\end{equation}
Substituting \eqref{eq:Q-factor-mn} with $r=m+n$ into \eqref{KFTH111}, we obtain \eqref{KFTH1}.
\end{proof}

For $b_1=0$ or $b_2=0$, formula \eqref{KFTH1} reduces to the first Miller--Paris transformation \eqref{MP1}.

\begin{remark}
In the univariate case, the upper parameters of a hypergeometric function enter the series coefficients symmetrically. For $F_1$, this symmetry is lost: the parameter $a$ enters through $(a)_{m+n}$, whereas $b_1$ and $b_2$ enter through $(b_1)_m$ and $(b_2)_n$, respectively. Thus, the argument leading to the second Miller--Paris transformation does not carry over directly to this setting.
\end{remark}

The next theorem concerns IPD pairs in the parameter groups associated separately with $m$ and $n$.

\begin{theorem}\label{thm:m-n-transform}
Let $p\in\mathbb N$, $\mathbf f=(f_1,\dots,f_p)\in\mathbb C^p$, $a,b_1,b_2,c,x,y\in\mathbb C$, $c,f_1,\dots,f_p\notin-\mathbb N_0$, $\lambda_1=c-b_1-b_2-p$, and $(\lambda_1)_p\ne0$.

If $b_1\ne f_j$ for $j=1,\dots,p$, then
\begin{equation}\label{eq:36}
F^{\,1:p+1;1}_{\,1:p;0}\!\left(
\begin{array}{c}
a:b_1,\mathbf f+\mathbf1;b_2\\
c:\mathbf f;-
\end{array};x,y\right)
=(1-x)^{-a}
F^{\,1:p+1;1}_{\,1:p;0}\!\left(
\begin{array}{c}
a:\lambda_1,\widetilde{\betta}+\mathbf1;b_2\\
c:\widetilde{\betta};-
\end{array};\frac{x}{x-1},\frac{y-x}{1-x}\right),
\end{equation}
where $\widetilde{\betta}=(\widetilde\eta_1,\dots,\widetilde\eta_p)$ is the vector of roots of the polynomial
\begin{equation}\label{eq:Qtilde}
\widetilde Q_p(t)=\sum_{k=0}^{p}\sigma_{p-k}(\mathbf f)
\sum_{l=0}^{k}S(k,l)(b_1)_l(\lambda_1-t)_{p-l}(t)_l.
\end{equation}
By symmetry,
if $b_2\ne f_j$ for $j=1,\dots,p$, then
\begin{equation}\label{eq:37}
F^{\,1:1;p+1}_{\,1:0;p}\!\left(
\begin{array}{c}
a:b_1;b_2,\mathbf f+\mathbf1\\
c:-;\mathbf f
\end{array};x,y\right)
=(1-y)^{-a}
F^{\,1:1;p+1}_{\,1:0;p}\!\left(
\begin{array}{c}
a:b_1;\lambda_1,\check{\betta}+\mathbf1\\
c:-;\check{\betta}
\end{array};\frac{x-y}{1-y},\frac{y}{y-1}\right),
\end{equation}
where $\check{\betta}=(\check\eta_1,\dots,\check\eta_p)$ is the vector of roots of
\begin{equation}\label{eq:Qcheck}
\check Q_p(t)=\sum_{k=0}^{p}\sigma_{p-k}(\mathbf f)
\sum_{l=0}^{k}S(k,l)(b_2)_l(\lambda_1-t)_{p-l}(t)_l.
\end{equation}
\end{theorem}

\begin{proof}
We first prove \eqref{eq:36}. Expanding the Kamp\'e de F\'eriet function into a double series, we obtain
\begin{equation}\label{eq:m-start}
\sum_{m=0}^{\infty}\sum_{n=0}^{\infty}
\frac{(a)_{m+n}(b_1)_m(b_2)_n}{(c)_{m+n}}
\frac{(\mathbf f+\mathbf1)_m}{(\mathbf f)_m}
\frac{x^m}{m!}\frac{y^n}{n!}.
\end{equation}
Here the IPD factor depends only on $m$ and has the expansion
\begin{equation}\label{eq:m-ipd-expand}
\frac{(\mathbf f+\mathbf1)_m}{(\mathbf f)_m}
=\frac{1}{(\mathbf f)_1}\sum_{k=0}^{p}\sigma_{p-k}(\mathbf f)m^k
=\frac{1}{(\mathbf f)_1}\sum_{k=0}^{p}\sigma_{p-k}(\mathbf f)
\sum_{l=0}^{k}S(k,l)[m]_l.
\end{equation}
Substituting \eqref{eq:m-ipd-expand} into \eqref{eq:m-start}, using $[m]_l/m!=1/(m-l)!$, and shifting $m\mapsto m+l$, we obtain
\begin{multline}\label{eq:KdF-m-F1}
F^{\,1:p+1;1}_{\,1:p;0}\!\left(
\begin{array}{c}
a:b_1,\mathbf f+\mathbf1;b_2\\
c:\mathbf f;-
\end{array};x,y\right)\\
=\frac{1}{(\mathbf f)_1}\sum_{k=0}^{p}\sigma_{p-k}(\mathbf f)
\sum_{l=0}^{k}S(k,l)
\frac{(a)_l(b_1)_l}{(c)_l}x^l
F_1(a+l;b_1+l,b_2;c+l;x,y).
\end{multline}
Apply \eqref{Sh25} to the function $F_1$:
\begin{multline}\label{eq:m-Sh25}
F_1(a+l;b_1+l,b_2;c+l;x,y)\\
=(1-x)^{-a-l}
F_1\!\left(a+l;c-b_1-b_2,b_2;c+l;
\frac{x}{x-1},\frac{y-x}{1-x}\right).
\end{multline}
Expand the right-hand side of \eqref{eq:m-Sh25} into a double series with summation indices $u,n\ge0$ and substitute it into \eqref{eq:KdF-m-F1}. Setting $m=u+l$, we use
\[
(m-l)!=(-1)^l\frac{m!}{(-m)_l},
\qquad
(c-b_1-b_2)_{m-l}
=\frac{(\lambda_1)_m(\lambda_1+m)_{p-l}}{(\lambda_1)_p},
\]
where $\lambda_1=c-b_1-b_2-p$. After cancellation, we obtain
\begin{multline}\label{eq:m-polynomial-stage}
F^{\,1:p+1;1}_{\,1:p;0}\!\left(
\begin{array}{c}
a:b_1,\mathbf f+\mathbf1;b_2\\
c:\mathbf f;-
\end{array};x,y\right)\\
=(1-x)^{-a}
\sum_{m=0}^{\infty}\sum_{n=0}^{\infty}
\frac{(a)_{m+n}(\lambda_1)_m(b_2)_n}{(c)_{m+n}m!n!}
\left(\frac{x}{x-1}\right)^m
\left(\frac{y-x}{1-x}\right)^n\\
\times\frac{1}{(\lambda_1)_p(\mathbf f)_1}
\sum_{k=0}^{p}\sigma_{p-k}(\mathbf f)
\sum_{l=0}^{k}S(k,l)(b_1)_l(\lambda_1+m)_{p-l}(-m)_l.
\end{multline}
Regarding the inner finite sum in \eqref{eq:m-polynomial-stage} as a polynomial in $m$ and substituting $m=-t$, we obtain
\[
\widetilde Q_p(t)=\sum_{k=0}^{p}\sigma_{p-k}(\mathbf f)
\sum_{l=0}^{k}S(k,l)(b_1)_l(\lambda_1-t)_{p-l}(t)_l,
\]
where $\widetilde Q_p$ is defined by \eqref{eq:Qtilde}. Thus, the last factor in \eqref{eq:m-polynomial-stage} is
\[
\frac{1}{(\lambda_1)_p(\mathbf f)_1}
\sum_{k=0}^{p}\sigma_{p-k}(\mathbf f)
\sum_{l=0}^{k}S(k,l)(b_1)_l(\lambda_1+m)_{p-l}(-m)_l=\frac{\widetilde Q_p(-m)}{(\lambda_1)_p(\mathbf f)_1}.
\]
At $t=0$, we have
\[
\widetilde Q_p(0)=(\lambda_1)_p(\mathbf f)_1.
\]
If $\widetilde\eta_1,\dots,\widetilde\eta_p$ are the roots of $\widetilde Q_p$, then
\[
\frac{\widetilde Q_p(-m)}{\widetilde Q_p(0)}
=\prod_{j=1}^{p}\frac{m+\widetilde\eta_j}{\widetilde\eta_j}
=\frac{(\widetilde{\betta}+\mathbf1)_m}{(\widetilde{\betta})_m}.
\]
Substituting this identity into \eqref{eq:m-polynomial-stage} yields the Kamp\'e de F\'eriet function on the right-hand side of \eqref{eq:36}.

We now prove \eqref{eq:37}. In this case, the IPD factor depends only on $n$:
\[
\frac{(\mathbf f+\mathbf1)_n}{(\mathbf f)_n}
=\frac{1}{(\mathbf f)_1}\sum_{k=0}^{p}\sigma_{p-k}(\mathbf f)
\sum_{l=0}^{k}S(k,l)[n]_l.
\]
Proceeding as above and shifting $n\mapsto n+l$, we obtain
\begin{multline}\label{eq:KdF-n-F1}
F^{\,1:1;p+1}_{\,1:0;p}\!\left(
\begin{array}{c}
a:b_1;b_2,\mathbf f+\mathbf1\\
c:-;\mathbf f
\end{array};x,y\right)\\
=\frac{1}{(\mathbf f)_1}\sum_{k=0}^{p}\sigma_{p-k}(\mathbf f)
\sum_{l=0}^{k}S(k,l)
\frac{(a)_l(b_2)_l}{(c)_l}y^l
F_1(a+l;b_1,b_2+l;c+l;x,y).
\end{multline}
Apply \eqref{Sh27} to $F_1$ in \eqref{eq:KdF-n-F1} and expand the result into a double series with summation indices $m,v\ge0$. Setting $n=v+l$ and using
\[
(n-l)!=(-1)^l\frac{n!}{(-n)_l},
\qquad
(c-b_1-b_2)_{n-l}
=\frac{(\lambda_1)_n(\lambda_1+n)_{p-l}}{(\lambda_1)_p},
\]
we obtain
\begin{multline}\label{eq:n-polynomial-stage}
F^{\,1:1;p+1}_{\,1:0;p}\!\left(
\begin{array}{c}
a:b_1;b_2,\mathbf f+\mathbf1\\
c:-;\mathbf f
\end{array};x,y\right)\\
=(1-y)^{-a}
\sum_{m=0}^{\infty}\sum_{n=0}^{\infty}
\frac{(a)_{m+n}(b_1)_m(\lambda_1)_n}{(c)_{m+n}m!n!}
\left(\frac{x-y}{1-y}\right)^m
\left(\frac{y}{y-1}\right)^n\\
\times\frac{1}{(\lambda_1)_p(\mathbf f)_1}
\sum_{k=0}^{p}\sigma_{p-k}(\mathbf f)
\sum_{l=0}^{k}S(k,l)(b_2)_l(\lambda_1+n)_{p-l}(-n)_l.
\end{multline}
Regarding the inner finite sum in \eqref{eq:n-polynomial-stage} as a polynomial in $n$ and substituting $n=-t$, we obtain
\[
\sum_{k=0}^{p}\sigma_{p-k}(\mathbf f)
\sum_{l=0}^{k}S(k,l)(b_2)_l(\lambda_1-t)_{p-l}(t)_l
=\check Q_p(t),
\]
where $\check Q_p$ is defined by \eqref{eq:Qcheck}. Thus, the last factor in \eqref{eq:n-polynomial-stage} is
\[
\frac{1}{(\lambda_1)_p(\mathbf f)_1}
\sum_{k=0}^{p}\sigma_{p-k}(\mathbf f)
\sum_{l=0}^{k}S(k,l)(b_2)_l(\lambda_1+n)_{p-l}(-n)_l=\frac{\check Q_p(-n)}{(\lambda_1)_p(\mathbf f)_1}.
\]
At $t=0$, we have
\[
\check Q_p(0)=(\lambda_1)_p(\mathbf f)_1.
\]
If $\check\eta_1,\dots,\check\eta_p$ are the roots of $\check Q_p$, then
\[
\frac{\check Q_p(-n)}{\check Q_p(0)}
=\prod_{j=1}^{p}\frac{n+\check\eta_j}{\check\eta_j}
=\frac{(\check{\betta}+\mathbf1)_n}{(\check{\betta})_n}.
\]
Substitution into \eqref{eq:n-polynomial-stage} gives \eqref{eq:37}.
\end{proof}

Setting $b_2=0$ in \eqref{eq:36} or $b_1=0$ in \eqref{eq:37} yields a form of the first Miller--Paris transformation with the corresponding characteristic polynomial.

\section{Reduction formulas for Kamp\'e de F\'eriet functions}

We first give identities relating $F_1$ to Kamp\'e de F\'eriet functions without additional IPD pairs. They follow by representing $F_1$ as a series of Gauss hypergeometric functions and applying the Euler--Pfaff transformations.

\begin{proposition}\label{prop:F1-KdF}
Let $a,b_1,b_2,c,x,y\in\mathbb C$ and $c,c-b_2\notin-\mathbb N_0$. Then the following identities hold:
\begin{equation}\label{eq:F1-KdF-explicit-1st}
F^{2:1;0}_{1:1;0}\!\left(
\begin{array}{c}
a,c-b_2:b_1;-\\
c:c-b_2;-
\end{array};\frac{x}{1-y},\frac{y}{y-1}\right)
=(1-y)^aF_1(a;b_1,b_2;c;x,y)
\end{equation}
and
\begin{equation}\label{eq:F1-KdF-explicit}
F^{1:2;1}_{1:1;0}\!\left(
\begin{array}{c}
c-b_2:a,b_1;c-a\\
c:c-b_2;-
\end{array};x,y\right)
=(1-y)^{a+b_2-c}F_1(a;b_1,b_2;c;x,y).
\end{equation}
\end{proposition}

\begin{proof}
Group the terms in the defining series for $F_1$ by the index $m$. Then
\begin{equation}\label{eq:F1-Gauss-rep}
F_1(a;b_1,b_2;c;x,y)
=\sum_{m=0}^{\infty}\frac{(a)_m(b_1)_m}{(c)_m\,m!}x^m
{}_2F_1\left(\begin{matrix}a+m,b_2\\c+m\end{matrix}\Big|y\right).
\end{equation}
Apply Pfaff's transformation to the Gauss function:
\[
{}_2F_1\left(\begin{matrix}a+m,b_2\\c+m\end{matrix}\Big|y\right)
=(1-y)^{-a-m}
{}_2F_1\left(\begin{matrix}a+m,c+m-b_2\\c+m\end{matrix}\Big|\frac{y}{y-1}\right).
\]
Substituting this identity into \eqref{eq:F1-Gauss-rep} and expanding the Gauss function into a series, we obtain the double series
\begin{multline}\label{eq:F1-double-pfaff}
F_1(a;b_1,b_2;c;x,y)\\
=(1-y)^{-a}\sum_{m=0}^{\infty}\sum_{n=0}^{\infty}
\frac{(a)_m(b_1)_m}{(c)_m\,m!}
\frac{(a+m)_n(c+m-b_2)_n}{(c+m)_n\,n!}
\left(\frac{x}{1-y}\right)^m
\left(\frac{y}{y-1}\right)^n.
\end{multline}
Using
\[
(a)_m(a+m)_n=(a)_{m+n},\qquad
(c)_m(c+m)_n=(c)_{m+n},\qquad
(c+m-b_2)_n=\frac{(c-b_2)_{m+n}}{(c-b_2)_m},
\]
we rewrite \eqref{eq:F1-double-pfaff} as
\[
(1-y)^{-a}\sum_{m=0}^{\infty}\sum_{n=0}^{\infty}
\frac{(a)_{m+n}(c-b_2)_{m+n}(b_1)_m}
{(c)_{m+n}(c-b_2)_m}
\frac{1}{m!n!}
\left(\frac{x}{1-y}\right)^m
\left(\frac{y}{y-1}\right)^n.
\]
By definition \eqref{eq:KDF}, this series is the left-hand side of \eqref{eq:F1-KdF-explicit-1st} multiplied by $(1-y)^{-a}$, which proves the first identity.

For the second identity, we again use \eqref{eq:F1-Gauss-rep}, this time applying Euler's transformation:
\[
{}_2F_1\left(\begin{matrix}a+m,b_2\\c+m\end{matrix}\Big|y\right)
=(1-y)^{c-a-b_2}
{}_2F_1\left(\begin{matrix}c-a,c+m-b_2\\c+m\end{matrix}\Big|y\right).
\]
Substituting this identity and expanding the resulting Gauss function into a series give
\begin{multline}\label{eq:F1-double-euler}
F_1(a;b_1,b_2;c;x,y)\\
=(1-y)^{c-a-b_2}
\sum_{m=0}^{\infty}\sum_{n=0}^{\infty}
\frac{(a)_m(b_1)_m(c-a)_n(c+m-b_2)_n}
{(c)_m(c+m)_n\,m!n!}x^my^n.
\end{multline}
The same identities for the Pochhammer symbols yield
\[
F_1(a;b_1,b_2;c;x,y)
=(1-y)^{c-a-b_2}
\sum_{m=0}^{\infty}\sum_{n=0}^{\infty}
\frac{(c-b_2)_{m+n}(a)_m(b_1)_m(c-a)_n}
{(c)_{m+n}(c-b_2)_m\,m!n!}x^my^n.
\]
The last double series is the Kamp\'e de F\'eriet function in \eqref{eq:F1-KdF-explicit}.
\end{proof}

\begin{remark}\label{rem:F1-KdF-sym}
Interchanging the pairs $(b_1,x)$ and $(b_2,y)$ in Proposition~\ref{prop:F1-KdF}, we obtain, under the conditions $c,c-b_1\notin-\mathbb N_0$,
\begin{equation}\label{eq:F1-KdF-explicit-x}
F^{2:0;1}_{1:0;1}\!\left(
\begin{array}{c}
a,c-b_1:-;b_2\\
c:-;c-b_1
\end{array};\frac{x}{x-1},\frac{y}{1-x}\right)
=(1-x)^{a}
F_1(a;b_1,b_2;c;x,y)
\end{equation}
and
\begin{equation}\label{eq:F1-KdF-explicit-x-2}
F^{1:1;2}_{1:0;1}\!\left(
\begin{array}{c}
c-b_1:c-a;a,b_2\\
c:-;c-b_1
\end{array};x,y\right)
=(1-x)^{a+b_1-c}
F_1(a;b_1,b_2;c;x,y).
\end{equation}
\end{remark}

The following results involve IPD pairs and give reductions to univariate generalized hypergeometric functions. The diagonal identities \eqref{TH3part1} and \eqref{Th4first} are direct consequences of Vandermonde's convolution and the first Miller--Paris transformation; we include them together with the reductions for IPD pairs associated with the individual summation indices.

\begin{theorem}\label{thm:reduction-I}
Let $p\in\mathbb N$, $\mathbf f=(f_1,\dots,f_p)\in\mathbb C^p$, $a,b_1,b_2,c,x\in\mathbb C$, $c,f_1,\dots,f_p\notin-\mathbb N_0$, $\lambda_1=c-b_1-b_2-p$, and $(\lambda_1)_p\ne0$.

If $b_1+b_2\ne f_j$ for $j=1,\dots,p$, then
\begin{equation}\label{TH3part1}
F_{p+1:0;0}^{p+1:1;1}\!\left(
\begin{array}{c}
a,\mathbf f+\mathbf1:b_1;b_2\\
c,\mathbf f:-;-
\end{array};x,x\right)
=(1-x)^{-a}
{}_{p+2}F_{p+1}\left(\begin{matrix}a,\lambda_1,\boldsymbol\mu+\mathbf1\\c,\boldsymbol\mu\end{matrix}\Big|\frac{x}{x-1}\right),
\end{equation}
where $\boldsymbol\mu=(\mu_1,\dots,\mu_p)$ is the vector of roots of
\begin{equation}\label{eq:Rhat}
\widehat R_p(t)=\sum_{k=0}^{p}\sigma_{p-k}(\mathbf f)
\sum_{l=0}^{k}S(k,l)(b_1+b_2)_l(t)_l(\lambda_1-t)_{p-l}.
\end{equation}

If $b_2\ne f_j$ for $j=1,\dots,p$, then
\begin{equation}\label{Th3part2}
F_{1:0;p}^{1:1;p+1}\!\left(
\begin{array}{c}
a:b_1;b_2,\mathbf f+\mathbf1\\
c:-;\mathbf f
\end{array};x,x\right)
=(1-x)^{-a}
{}_{p+2}F_{p+1}\left(\begin{matrix}a,\lambda_1,\boldsymbol\nu+\mathbf1\\c,\boldsymbol\nu\end{matrix}\Big|\frac{x}{x-1}\right),
\end{equation}
where $\boldsymbol\nu=(\nu_1,\dots,\nu_p)$ is the vector of roots of
\begin{equation}\label{eq:Rtilde}
\widetilde R_p(t)=\sum_{k=0}^{p}\sigma_{p-k}(\mathbf f)
\sum_{l=0}^{k}S(k,l)(b_2)_l(t)_l(\lambda_1-t)_{p-l}.
\end{equation}
\end{theorem}

\begin{proof}
We first prove \eqref{TH3part1}. By the definition of the
Kamp\'e de F\'eriet function,
\begin{equation}\label{eq:redI-double}
\begin{gathered}
F_{p+1:0;0}^{p+1:1;1}\!\left(
\begin{array}{c}
a,\mathbf f+\mathbf1:b_1;b_2\\
c,\mathbf f:-;-
\end{array};x,x\right)
=\sum_{m=0}^{\infty}\sum_{n=0}^{\infty}
\frac{(a)_{m+n}(\mathbf f+\mathbf1)_{m+n}(b_1)_m(b_2)_n}
{(c)_{m+n}(\mathbf f)_{m+n}}
\frac{x^{m+n}}{m!\,n!}.
\end{gathered}
\end{equation}
For $|x|<1$, this series converges absolutely. We group its
terms according to $r=m+n$. For each fixed $r$, we have
$n=r-m$ and $m=0,\dots,r$, so
\begin{multline*}
\sum_{m=0}^{\infty}\sum_{n=0}^{\infty}
\frac{(a)_{m+n}(\mathbf f+\mathbf1)_{m+n}(b_1)_m(b_2)_n}
{(c)_{m+n}(\mathbf f)_{m+n}}
\frac{x^{m+n}}{m!\,n!}=
\\\sum_{r=0}^{\infty}\sum_{m=0}^{r}
\frac{(a)_r(\mathbf f+\mathbf1)_r(b_1)_m(b_2)_{r-m}}
{(c)_r(\mathbf f)_r}
\frac{x^r}{m!\,(r-m)!}=
\\\sum_{r=0}^{\infty}
\frac{(a)_r(\mathbf f+\mathbf1)_r}
{(c)_r(\mathbf f)_r}\,x^r
\sum_{m=0}^{r}
\frac{(b_1)_m(b_2)_{r-m}}{m!\,(r-m)!}.
\end{multline*}

By Vandermonde's convolution,
\begin{equation}\label{eq:vandermonde-reduction}
\sum_{m=0}^{r}
\frac{(b_1)_m(b_2)_{r-m}}{m!\,(r-m)!}
=\frac{(b_1+b_2)_r}{r!}.
\end{equation}
Substituting \eqref{eq:vandermonde-reduction} into the
preceding expression and using the defining series of the
generalized hypergeometric function, we obtain
\begin{equation}\label{eq:redI-single}
F_{p+1:0;0}^{p+1:1;1}\!\left(
\begin{array}{c}
a,\mathbf f+\mathbf1:b_1;b_2\\
c,\mathbf f:-;-
\end{array};x,x\right)
=
{}_{p+2}F_{p+1}\!\left(
\begin{matrix}
a,b_1+b_2,\mathbf f+\mathbf1\\
c,\mathbf f
\end{matrix}\Big|x\right).
\end{equation}

We now apply the first Miller--Paris transformation
\eqref{MP1} to the right-hand side of \eqref{eq:redI-single}
with $b=b_1+b_2$. Then $\lambda=c-(b_1+b_2)-p=\lambda_1$,
and the characteristic polynomial defined by
\eqref{MP1Pol} satisfies
\[
\begin{aligned}
\left.M_p(t)\right|_{b=b_1+b_2,\;\lambda=\lambda_1}
=\sum_{k=0}^{p}\sigma_{p-k}(\mathbf f)
\sum_{l=0}^{k}S(k,l)(b_1+b_2)_l(t)_l
(\lambda_1-t)_{p-l}
=\widehat R_p(t),
\end{aligned}
\]
where the last equality follows from \eqref{eq:Rhat}.
Thus, the roots of the characteristic polynomial are
$\mu_1,\dots,\mu_p$, and \eqref{MP1} yields
\eqref{TH3part1}.

We turn to \eqref{Th3part2}. Setting $y=x$ in transformation \eqref{eq:37} gives
\begin{equation}\label{eq:redI-second-kdf}
F_{1:0;p}^{1:1;p+1}\!\left(
\begin{array}{c}
a:b_1;b_2,\mathbf f+\mathbf1\\
c:-;\mathbf f
\end{array};x,x\right)
=(1-x)^{-a}
F_{1:0;p}^{1:1;p+1}\!\left(
\begin{array}{c}
a:b_1;\lambda_1,\check{\betta}+\mathbf1\\
c:-;\check{\betta}
\end{array};0,\frac{x}{x-1}\right).
\end{equation}
Expand the function on the right into a double series. Since its first argument is zero, all terms with $m\ge1$ vanish, leaving
\[
\sum_{n=0}^{\infty}
\frac{(a)_n(\lambda_1)_n(\check{\betta}+\mathbf1)_n}
{(c)_n(\check{\betta})_n\,n!}
\left(\frac{x}{x-1}\right)^n
={}_{p+2}F_{p+1}\left(
\begin{matrix}a,\lambda_1,\check{\betta}+\mathbf1\\c,\check{\betta}\end{matrix}\Big|\frac{x}{x-1}\right).
\]
In transformation \eqref{eq:37}, the vector
$\check{\boldsymbol\eta}$ consists of the roots of
$\check Q_p(t)$. By \eqref{eq:Qcheck} and \eqref{eq:Rtilde},
\[
\begin{aligned}
\check Q_p(t)
=\sum_{k=0}^{p}\sigma_{p-k}(\mathbf f)
\sum_{l=0}^{k}S(k,l)(b_2)_l(t)_l
(\lambda_1-t)_{p-l}
=\widetilde R_p(t).
\end{aligned}
\]
Hence, $\check{\boldsymbol\eta}$ coincides with
$\boldsymbol\nu$ up to a permutation of its components.
Since the generalized hypergeometric function is invariant
under such a permutation of the parameter pairs,
\eqref{Th3part2} follows.
\end{proof}

\begin{theorem}\label{thm:reduction-II}
Let $p\in\mathbb N$, $\mathbf f=(f_1,\dots,f_p)\in\mathbb C^p$, $a,b_1,b_2,c,y\in\mathbb C$, $c,f_1,\dots,f_p\notin-\mathbb N_0$, $\lambda=c-a-p$, and $(\lambda)_p\ne0$.

If $a\ne f_j$ for $j=1,\dots,p$, then
\begin{equation}\label{Th4first}
F_{p+1:0;0}^{p+1:1;1}\!\left(
\begin{array}{c}
a,\mathbf f+\mathbf1:b_1;b_2\\
c,\mathbf f:-;-
\end{array};y,y\right)
=(1-y)^{-b_1-b_2}
{}_{p+2}F_{p+1}\left(\begin{matrix}\lambda,b_1+b_2,\widetilde{\boldsymbol\mu}+\mathbf1\\c,\widetilde{\boldsymbol\mu}\end{matrix}\Big|\frac{y}{y-1}\right),
\end{equation}
where $\widetilde{\boldsymbol\mu}=(\widetilde\mu_1,\dots,\widetilde\mu_p)$ is the vector of roots of
\begin{equation}\label{eq:Qmu-tilde}
\widetilde T_p(t)=\sum_{k=0}^{p}\sigma_{p-k}(\mathbf f)
\sum_{l=0}^{k}S(k,l)(a)_l(\lambda-t)_{p-l}(t)_l.
\end{equation}

Furthermore, if $(b_1+b_2)_p\ne0$ and the polynomial $\widehat T_p$ defined below has degree $p$, then
\begin{equation}\label{Th4second}
F_{1:p;0}^{1:p+1;1}\!\left(
\begin{array}{c}
a:b_1,\mathbf f+\mathbf1;b_2\\
c:\mathbf f;-
\end{array};y,y\right)
=(1-y)^{-b_1-b_2}
{}_{p+2}F_{p+1}\left(\begin{matrix}\lambda,b_1+b_2,\widehat{\boldsymbol\mu}+\mathbf1\\c,\widehat{\boldsymbol\mu}\end{matrix}\Big|\frac{y}{y-1}\right),
\end{equation}
where $\widehat{\boldsymbol\mu}=(\widehat\mu_1,\dots,\widehat\mu_p)$ is the vector of roots of
\begin{equation}\label{eq:Qmu-hat}
\widehat T_p(t)=\sum_{k=0}^{p}\sigma_{p-k}(\mathbf f)
\sum_{l=0}^{k}S(k,l)\frac{(a)_l(b_1)_l}{(b_1+b_2)_l}(\lambda-t)_{p-l}(t)_l.
\end{equation}
If $(b_1+b_2)_p\ne0$ and the polynomial $\check T_p$ defined below has degree $p$, then
\begin{equation}\label{Th4third}
F_{1:0;p}^{1:1;p+1}\!\left(
\begin{array}{c}
a:b_1;b_2,\mathbf f+\mathbf1\\
c:-;\mathbf f
\end{array};y,y\right)
=(1-y)^{-b_1-b_2}
{}_{p+2}F_{p+1}\left(\begin{matrix}\lambda,b_1+b_2,\check{\boldsymbol\mu}+\mathbf1\\c,\check{\boldsymbol\mu}\end{matrix}\Big|\frac{y}{y-1}\right),
\end{equation}
where $\check{\boldsymbol\mu}=(\check\mu_1,\dots,\check\mu_p)$ is the vector of roots of
\begin{equation}\label{eq:Qmu-check}
\check T_p(t)=\sum_{k=0}^{p}\sigma_{p-k}(\mathbf f)
\sum_{l=0}^{k}S(k,l)\frac{(a)_l(b_2)_l}{(b_1+b_2)_l}(\lambda-t)_{p-l}(t)_l.
\end{equation}
\end{theorem}

\begin{proof}
We begin with \eqref{Th4first}. Setting $x=y$ in the
diagonal reduction \eqref{eq:redI-single}, we obtain
\begin{equation}\label{eq:redII-first-single}
F_{p+1:0;0}^{p+1:1;1}\!\left(
\begin{array}{c}
a,\mathbf f+\mathbf1:b_1;b_2\\
c,\mathbf f:-;-
\end{array};y,y\right)
=
{}_{p+2}F_{p+1}\!\left(
\begin{matrix}
a,b_1+b_2,\mathbf f+\mathbf1\\
c,\mathbf f
\end{matrix}\Big|y\right).
\end{equation}

By symmetry in the first two upper parameters, we may
apply the first Miller--Paris transformation \eqref{MP1}
with $a$ replaced by $b_1+b_2$ and $b$ replaced by $a$.
The transformed parameter is then $\lambda=c-a-p$.
For the characteristic polynomial defined by
\eqref{MP1Pol}, this substitution gives
\[
\begin{aligned}
\left.M_p(t)\right|_{b=a,\;\lambda=c-a-p}
=\sum_{k=0}^{p}\sigma_{p-k}(\mathbf f)
\sum_{l=0}^{k}S(k,l)(a)_l(t)_l
(\lambda-t)_{p-l}
=\widetilde T_p(t),
\end{aligned}
\]
where the last equality follows from \eqref{eq:Qmu-tilde}.
Thus, the roots of the characteristic polynomial are
$\widetilde\mu_1,\dots,\widetilde\mu_p$, and
\eqref{MP1} yields
\[
\begin{gathered}
{}_{p+2}F_{p+1}\!\left(
\begin{matrix}
a,b_1+b_2,\mathbf f+\mathbf1\\
c,\mathbf f
\end{matrix}\Big|y\right)
=(1-y)^{-b_1-b_2}
{}_{p+2}F_{p+1}\!\left(
\begin{matrix}
\lambda,b_1+b_2,\widetilde{\boldsymbol\mu}+\mathbf1\\
c,\widetilde{\boldsymbol\mu}
\end{matrix}\Big|\frac{y}{y-1}\right).
\end{gathered}
\]
Combining this identity with \eqref{eq:redII-first-single}
proves \eqref{Th4first}.

To prove \eqref{Th4second}, use the representation \eqref{eq:KdF-m-F1} obtained above and set $x=y$:
\begin{multline}\label{eq:redII-m-F1}
F_{1:p;0}^{1:p+1;1}\!\left(
\begin{array}{c}
a:b_1,\mathbf f+\mathbf1;b_2\\
c:\mathbf f;-
\end{array};y,y\right)\\
=\frac{1}{(\mathbf f)_1}\sum_{k=0}^{p}\sigma_{p-k}(\mathbf f)
\sum_{l=0}^{k}S(k,l)
\frac{(a)_l(b_1)_l}{(c)_l}y^l
F_1(a+l;b_1+l,b_2;c+l;y,y).
\end{multline}
Apply \eqref{Sh28} to $F_1$. Since both original arguments are equal to $y$, the second argument of the transformed function is zero, and
\begin{equation}\label{eq:redII-Sh28}
F_1(a+l;b_1+l,b_2;c+l;y,y)
=(1-y)^{c-a-b_1-b_2-l}
{}_2F_1\left(\begin{matrix}c-a,c-b_1-b_2\\c+l\end{matrix}\Big|y\right).
\end{equation}
We now apply Pfaff's transformation to the Gauss function in \eqref{eq:redII-Sh28}:
\[
{}_2F_1\left(\begin{matrix}c-a,c-b_1-b_2\\c+l\end{matrix}\Big|y\right)
=(1-y)^{a-c}
{}_2F_1\left(\begin{matrix}c-a,b_1+b_2+l\\c+l\end{matrix}\Big|\frac{y}{y-1}\right).
\]
Substituting this into \eqref{eq:redII-m-F1}, we obtain
\begin{multline}\label{eq:redII-before-shift}
F_{1:p;0}^{1:p+1;1}\!\left(
\begin{array}{c}
a:b_1,\mathbf f+\mathbf1;b_2\\
c:\mathbf f;-
\end{array};y,y\right)\\
=(1-y)^{-b_1-b_2}\frac{1}{(\mathbf f)_1}
\sum_{k=0}^{p}\sigma_{p-k}(\mathbf f)
\sum_{l=0}^{k}S(k,l)
\frac{(a)_l(b_1)_l}{(c)_l}
\left(\frac{y}{1-y}\right)^l\\
\times{}_2F_1\left(\begin{matrix}c-a,b_1+b_2+l\\c+l\end{matrix}\Big|\frac{y}{y-1}\right).
\end{multline}
Expand ${}_2F_1$ into a series with summation index $s\ge0$ and set $r=s+l$, so that $s=r-l$. We use
\[
(r-l)!=(-1)^l\frac{r!}{(-r)_l},\qquad
(c-a)_{r-l}=\frac{(\lambda)_r(\lambda+r)_{p-l}}{(\lambda)_p},
\]
as well as
\[
(b_1+b_2+l)_{r-l}=\frac{(b_1+b_2)_r}{(b_1+b_2)_l},
\qquad
(c+l)_{r-l}=\frac{(c)_r}{(c)_l}.
\]
Thus, \eqref{eq:redII-before-shift} takes the form
\begin{multline}\label{eq:Th4second-stage}
F_{1:p;0}^{1:p+1;1}\!\left(
\begin{array}{c}
a:b_1,\mathbf f+\mathbf1;b_2\\
c:\mathbf f;-
\end{array};y,y\right)
=(1-y)^{-b_1-b_2}
\sum_{r=0}^{\infty}
\frac{(\lambda)_r(b_1+b_2)_r}{(c)_r\,r!}
\left(\frac{y}{y-1}\right)^r\\
\times\frac{1}{(\lambda)_p(\mathbf f)_1}
\sum_{k=0}^{p}\sigma_{p-k}(\mathbf f)
\sum_{l=0}^{k}S(k,l)
\frac{(a)_l(b_1)_l}{(b_1+b_2)_l}
(\lambda+r)_{p-l}(-r)_l.
\end{multline}
Regarding the inner finite sum in \eqref{eq:Th4second-stage} as a polynomial in $r$ and substituting $r=-t$, we obtain
\[
\sum_{k=0}^{p}\sigma_{p-k}(\mathbf f)
\sum_{l=0}^{k}S(k,l)
\frac{(a)_l(b_1)_l}{(b_1+b_2)_l}
(\lambda-t)_{p-l}(t)_l,
\]
which is $\widehat T_p(t)$ defined by \eqref{eq:Qmu-hat}. At $t=0$, we have
\[
\widehat T_p(0)=(\lambda)_p(\mathbf f)_1.
\]
If $\widehat\mu_1,\dots,\widehat\mu_p$ are the roots of $\widehat T_p$, then
\[
\frac{\widehat T_p(-r)}{\widehat T_p(0)}
=\prod_{j=1}^{p}\frac{r+\widehat\mu_j}{\widehat\mu_j}
=\frac{(\widehat{\boldsymbol\mu}+\mathbf1)_r}{(\widehat{\boldsymbol\mu})_r}.
\]
Substitution into \eqref{eq:Th4second-stage} gives \eqref{Th4second}.

To prove \eqref{Th4third}, set $x=y$ in \eqref{eq:KdF-n-F1}:
\begin{multline}\label{eq:redII-n-F1}
F_{1:0;p}^{1:1;p+1}\!\left(
\begin{array}{c}
a:b_1;b_2,\mathbf f+\mathbf1\\
c:-;\mathbf f
\end{array};y,y\right)\\
=\frac{1}{(\mathbf f)_1}\sum_{k=0}^{p}\sigma_{p-k}(\mathbf f)
\sum_{l=0}^{k}S(k,l)
\frac{(a)_l(b_2)_l}{(c)_l}y^l
F_1(a+l;b_1,b_2+l;c+l;y,y).
\end{multline}
Applying transformation \eqref{Sh29}, we obtain
\[
F_1(a+l;b_1,b_2+l;c+l;y,y)
=(1-y)^{c-a-b_1-b_2-l}
{}_2F_1\left(\begin{matrix}c-a,c-b_1-b_2\\c+l\end{matrix}\Big|y\right).
\]
A further application of Pfaff's transformation gives the same Gauss function as in the proof of \eqref{Th4second}. Expanding it into a series with summation index $s\ge0$ and setting $r=s+l$, we obtain
\begin{multline}\label{eq:Th4third-stage}
F_{1:0;p}^{1:1;p+1}\!\left(
\begin{array}{c}
a:b_1;b_2,\mathbf f+\mathbf1\\
c:-;\mathbf f
\end{array};y,y\right)
=(1-y)^{-b_1-b_2}
\sum_{r=0}^{\infty}
\frac{(\lambda)_r(b_1+b_2)_r}{(c)_r\,r!}
\left(\frac{y}{y-1}\right)^r\\
\times\frac{1}{(\lambda)_p(\mathbf f)_1}
\sum_{k=0}^{p}\sigma_{p-k}(\mathbf f)
\sum_{l=0}^{k}S(k,l)
\frac{(a)_l(b_2)_l}{(b_1+b_2)_l}
(\lambda+r)_{p-l}(-r)_l.
\end{multline}
Regarding the inner finite sum in \eqref{eq:Th4third-stage} as a polynomial in $r$ and substituting $r=-t$, we obtain $\check T_p(t)$ defined by \eqref{eq:Qmu-check}. At $t=0$, we have
\[
\check T_p(0)=(\lambda)_p(\mathbf f)_1.
\]
If $\check\mu_1,\dots,\check\mu_p$ are the roots of $\check T_p$, then
\[
\frac{\check T_p(-r)}{\check T_p(0)}
=\prod_{j=1}^{p}\frac{r+\check\mu_j}{\check\mu_j}
=\frac{(\check{\boldsymbol\mu}+\mathbf1)_r}{(\check{\boldsymbol\mu})_r}.
\]
Substitution into \eqref{eq:Th4third-stage} gives \eqref{Th4third}.
\end{proof}

\section*{Funding}

The research was carried out within the state assignment for IAM FEB RAS (No.~075-00460-26-00).

This work was also financially supported by the Ministry of Science and Higher Education of the Russian Federation under Agreement No.~075-02-2026-1312 dated February~20, 2026, for the implementation of development programs for regional mathematical research and education centers (\href{https://www.dcmi.ru/}{Far Eastern Center of Mathematical Research}, Far Eastern Federal University).

\end{document}